\documentclass[11pt,fleqn]{article}
\usepackage[paper=a4paper]
  {geometry}

\usepackage[small]{titlesec}
\usepackage{paragraphs}

\usepackage{hyperref}

\usepackage{amsthm,thmtools}

\usepackage[utf8]{inputenc}
\usepackage[english]{babel}
\usepackage{enumerate}
\usepackage[osf,noBBpl]{mathpazo}
\usepackage[alphabetic,initials]{amsrefs}
\usepackage{amsfonts,amssymb,amsmath}
\usepackage{mathtools}
\usepackage{graphicx}
\usepackage[poly,arrow,curve,matrix]{xy}
\usepackage{wrapfig}
\usepackage{xcolor}
\usepackage{helvet}
\usepackage{stmaryrd}
\usepackage[normalem]{ulem}

\newskip\paraskip
\declaretheoremstyle[headformat=swapnumber, spaceabove=\paraskip,
bodyfont=\itshape]{mystyle}

\declaretheorem[name=Theorem, sibling=para, style=mystyle]{Theorem}

\declaretheoremstyle[numbered=no, spaceabove=\paraskip,
bodyfont=\itshape]{mystyle-empty}
\declaretheorem[name=Lemma, style=mystyle-empty]{Lemma*}
\declaretheorem[name=Proposition, style=mystyle-empty]{Proposition*}
\declaretheorem[name=Theorem, style=mystyle-empty]{Theorem*}
\declaretheorem[name=Corollary, style=mystyle-empty]{Corollary*}
\declaretheorem[name=Definition, style=mystyle-empty]{Definition*}
\declaretheorem[name=Example, style=mystyle-empty]{Example*}
\declaretheorem[name=Remark, style=mystyle-empty]{Remark*}

\declaretheoremstyle[
        headformat={{\bfseries\NUMBER.}{ \bfseries\NAME}},
        spaceabove=\paraskip, 
        headpunct={. },
        headfont=\bfseries,
        bodyfont=\normalfont
        ]{mystyle-plain}

\declaretheorem[name=Remark, sibling=para, style=mystyle-plain]{Remark}

\makeatletter
\renewenvironment{proof}[1][\textit{Proof}]{\par
  \pushQED{\qed}%
  \normalfont \topsep.75\paraskip\relax
  \trivlist
  \item[\hskip\labelsep
        \itshape
    #1\@addpunct{.}]\ignorespaces
}{%
  \popQED\endtrivlist\@endpefalse
}
\makeatother

\usepackage{tikz}
\usepackage{mathdots}
\newcommand\NN{\mathbb N}
\newcommand\CC{\mathbb C}

\newcommand\ZZ{\mathbb Z}

\newcommand\D{\overline D}
\newcommand\DD{\mathcal D}
\newcommand\N{\mathcal N}

\renewcommand\O{\mathcal O}

\newcommand\ot{\otimes}
\renewcommand\to{\longrightarrow}
\renewcommand\phi{\varphi}
\newcommand\id{\mathsf{id}}

\newcommand\g{\mathfrak g}
\newcommand\p{\mathfrak p}
\newcommand\m{\mathfrak m}
\newcommand\gl{\mathfrak{gl}}

\newcommand\std{\mathsf{std}}

\newcommand\interval[1]{\llbracket #1 \rrbracket}
\newcommand\Shuffle{\mathsf{Shuffle}}

\DeclareMathOperator\Specm{Specm}

\DeclareMathOperator\sym{sym}
\DeclareMathOperator\asym{asym}
\DeclareMathOperator\sg{sg}
\DeclareMathOperator\st{\mathsf{st}}

\newcommand\bigmodule{big GT module}

\title{Gelfand-Tsetlin modules over $\gl(n)$ with arbitrary characters}

\author{L.E. Ram\'irez\footnote{Universidade Federal do ABC, Santo Andr\'e-SP, 
Brasil \texttt{email:} luis.enrique@ufabc.edu.br} , 
P. Zadunaisky\footnote{Instituto de Matem\'atica e Estat\'istica, Universidade 
de S\~ao Paulo,  S\~ao Paulo SP, Brasil. \texttt{email:} pzadun@ime.usp.br.
The author is a FAPESP PostDoc Fellow, grant: 2016-25984-1 
S\~ao Paulo Research Foundation (FAPESP).}
}
\date{}

\begin{document}
\maketitle

\begin{abstract}
A Gelfand-Tsetlin tableau $T(v)$ induces a character $\chi_v$ of the 
Gelfand-Tsetlin subalgebra $\Gamma$ of $U = U(\gl(n,\CC))$. By a theorem due 
to Ovsienko, for each tableau $T(v)$ there exists a finite number of 
nonisomorphic irreducible Gelfand-Tsetlin modules with $\chi_v$ in its 
support, though explicit examples of such modules are only known for special
families of characters. In this article we build a family of Gelfand-Tsetlin
modules parametrized by characters, such that each character appears in its
corresponding module. We also find the support of these modules, with 
multiplicities.
\end{abstract}
\noindent\textbf{MSC 2010 Classification:} 17B10.\\
\noindent\textbf{Keywords:} Gelfand-Tsetlin modules, Gelfand-Tsetlin bases,
tableaux realization.

\section{Introduction}
The notion of a Gelfand-Tsetlin module (see Definition \ref{D:gt-module}) has 
its origin in the classical article \cite{GT-modules}, where I. Gelfand and M. 
Tsetlin gave an explicit presentation of all finite dimensional irreducible 
representations of $\g = \gl(n,\CC)$ in terms of certain combinatorial 
objects, which have come to be known as \emph{Gelfand-Tsetlin tableaux}, or 
GT tableaux for short. A GT tableau is a triangular array of $\frac{n(n+1)}{2}$
complex numbers, with $k$ entries in the $k$-th row; given a point $v \in 
\CC^{\frac{n(n+1)}{2}}$ we denote the corresponding array by $T(v)$. The group 
$G = S_1 \times S_2 \times \cdots \times S_n$ acts on the set of all tableaux, 
with $S_k$ permuting the elements in the $k$-th row. Gelfand and Tsetlin's
theorem establishes that any finite dimensional irreducible representation of 
$\g$ has a basis parameterized by GT tableaux with integer entries satisfying 
certain betweenness relations. Identifying the elements of the basis with the 
corresponding GT tableaux, the action of an element of $\g$ over a 
tableau is given by rational functions in its entries. These rational 
functions are known as the Gelfand-Tsetlin formulas; their poles form an 
infinite hyperplane array in $\CC^{\frac{n(n+1)}{2}}$.

The enveloping algebra $U = U(\g)$ contains a large (indeed, maximal) 
commutative subalgebra $\Gamma$ called the \emph{Gelfand-Tsetlin} subalgebra 
of $U$. A \emph{Gelfand-Tsetlin module} is a $U$-module that can be
decomposed as the direct sum of generalized eigenspaces for $\Gamma$.
The characters of $\Gamma$ are in one-to-one correspondence with GT tableaux 
modulo the action of $G$ (see \cite{Zh-compact-book}), and in the original 
construction of Gelfand and Tsetlin each tableau $T(v)$ is an eigenvector of 
$\Gamma$ whose eigenvalue is precisely the character $\chi_v: \Gamma \to \CC$ 
corresponding to $v$. Since no two tableaux in this construction are in the 
same $G$-orbit, the multiplicity of this character (i.e. the number of 
eigenvectors of eigenvalue $\chi_v$) is one.

Ovsienko proved in \cites{Ovs-finiteness, Ovs-strongly-nilpotent} that for 
each character $\chi: \Gamma \to \CC$ there exists a nonzero finite number of 
Gelfand-Tsetlin $U$-modules with $\chi$ in its character support. Many such 
modules have been constructed for different classes of characters, such as 
standard \cite{GT-modules}, generic \cite{DFO-GT-modules}, $1$-singular 
\cite{FGR-1-singular}, index $2$ \cite{FGR-2-index}, etc. However, no explicit 
construction of such modules is known for arbitrary characters. 

The first work in this direction is due to Y. Drozd, S. Ovsienko and V. 
Futorny, who introduced a large family of infinite dimensional $\g$-modules 
in \cite{DFO-GT-modules}. These GT modules have a basis parameterized by 
Gelfand-Tsetlin tableaux with complex coefficients such that no pattern is a 
pole for the rational functions appearing in the GT formulae (such tableaux 
are called \emph{generic}, hence the name ``generic Gelfand-Tsetlin module''). 
While each character in the decomposition of a generic Gelfand-Tsetlin module
appears with multiplicity one, there are examples of non-generic GT modules
with higher multiplicities. These examples were first encountered in 
\cites{Fut-generalization-Verma, Fut-semiprimitive} for $\mathfrak{sl}(3)$.

In \cite{FGR-1-singular} V. Futorny, D. Grantcharov and the first named author
constructed a GT module with $1$-singular characters, i.e. characters 
associated to tableaux over which the Gelfand-Tsetlin formulas may have 
singularities of order at most $1$. These modules have a basis in terms of 
so-called \emph{derived tableaux}, new objects which, according to the authors 
``are not new combinatorial objects'' but rather formal objects in a large
vector-space that contains classical GT tableaux. This construction was 
expanded and refined in the articles \cites{FGR-2-index, Zad-1-sing, 
V-geometric-singular-GT} for characters with more general singularities. The 
aim of this article is to extend this construction to \emph{arbitrary} 
characters and calculate the mutliplicity of the corresponding characters. 
This is achieved in section \ref{SINGULAR-GT}; in the process we 
disprove the statement above and give a combinatorial interpretation of 
derived tableaux. 

The general idea of our construction is the following. Let $K$ be the field of 
rational functions over the space of GT tableaux, and denote by $V$ the $K$ 
vector-space of arbitrary integral GT tableaux. This is a $U$-module, which we
call the \bigmodule, with the action of $\g$ given by the Gelfand-Tsetlin 
formulas. These rational functions lie in the algebra $A$ of regular functions 
over generic tableaux, and hence the $A$-lattice $L_A$ whose $A$-basis is the 
set of all integral tableaux is a $U$-submodule of $V$; now given a generic 
tableau $T(v)$, we can recover the corresponding generic module by 
specializing $V_A$ at $v$. This idea breaks down if $T(v)$ is a singular 
tableau, and in that case we replace $A$ with an algebra $B \subset K$ such 
that $(1)$ evaluation at $v$ makes sense and $(2)$ there exists a $B$-lattice 
$L_B \subset V$ which is also a $U$-submodule. Once this is done, the rest of 
the construction follows as in the generic case.

Denote by $\mu$ the composition of $\frac{n(n+1)}{2}$ given by 
$(1,2,\ldots, n)$. Each point $v \in \CC^{\frac{n(n+1)}{2}}$, or rather its 
class modulo $G$, defines a refinement $\eta(v)$ of $\mu$, and it turns out 
that the structure of the associated GT module $V(T(v))$ depends heavily on 
$\eta(v)$. While it is possible in principle to choose an algebra $B$ and a 
lattice $L_B$ that works for all $v$ simultaneously, we get more information 
by fixing a refinement $\eta$ and focusing on characters with $\eta(v) = 
\eta$, thus obtaining an algebra $B_\eta$ and a $B_\eta$-lattice $L_\eta$. 
Each $L_\eta$ has a basis of derived tableaux, and changing $\eta$ changes 
this basis in an essential way.

As shown in \cite{FGR-generic-irreducible}, the generic GT modules from 
\cite{DFO-GT-modules} are universal, in the sense that any irreducible GT 
module with a generic character in its support is isomorphic to a subquotient
of the corresponding generic GT module. The $1$-singular GT modules built in 
\cite{FGR-1-singular} are also universal with respect to $1$-singular 
characters, see \cite{FGR-drinfeld}. Since these modules are special cases of
our construction, we expect the singular GT modules built in this article to
be universal with respect to the characters in their support.

While finishing this paper the article \cite{V-geometric-singular-GT} by 
E. Vishnyakova was uploaded to the ArXiv, containing a similar construction
of $p$-singular GT modules, where $p \in \NN_{\geq 2}$. This is a special 
class of singular modules, associated to classes $v \in 
\CC^{\frac{n(n+1)}{2}}/G$ where the composition $\eta(v)$ has only one 
nontrivial part, which is equal to $p$.

\bigskip

The article is organized as follows. In section \ref{PREELIMINARIES} we 
set the notation used for the combinatorial invariants associated to tableaux. 
In section \ref{GT-MODULES} we review some basic facts on GT modules, including
the construction of generic and $1$-singular GT modules in terms of the
\bigmodule. In section \ref{SDD-OPERATORS} we study certain operators related
to divided differences which play a central role in the construction of the 
lattices $L_\eta$. Finally in section \ref{SINGULAR-GT} we present the lattices
$L_\eta$, build the GT modules associated to a character $\chi_v$ and find
its character support with the corresponding multiplicites.

\bigskip

\section{Preeliminaries}
\label{PREELIMINARIES}

\paragraph
Let $n,m \in \NN$. We write $\interval{n,m} = \{k \in \NN \mid n \leq k \leq 
m\}$ and $\interval{n} = \interval{1,n}$. We denote by $S_n$ the symmetric 
group on $n$ elements.
Recall that for each $\sigma \in S_n$ the length of $\sigma$, denoted by 
$\ell(\sigma)$, is the number of inversions of $\sigma$, i.e. the number
of pairs $(i,j) \in \interval{n}^2$ such that $i<j$ but $\sigma(i) > 
\sigma(j)$. There exists a unique longest word $w_0 \in S_n$ such that
$\ell(w_0) = \frac{n(n-1)}{2}$. Also $\ell(\sigma) = \ell(\sigma^{-1})$, and 
$\ell(\sigma^{-1} w_0) = \ell(w_0) - \ell(\sigma)$. For each $i \in 
\interval{n-1}$ the $i$-th simple transposition is $s_i = (i,i+1) \in S_n$. 
Simple transpositions generate $S_n$, and the length of $\sigma \in S_n$ is 
the minimal $l$ such that $\sigma$ can be written as $s_{i_1} s_{i_2} \cdots 
s_{i_l}$. Any such writing is called a reduced decomposition of $\sigma$.

\paragraph
\textbf{Compositions.}
\label{compositions}
Recall that a composition of $n$ is a sequence $\mu = (\mu_1, \ldots, \mu_r)$
of positive integers such that $\sum_{i=1}^r \mu_i = n$. The $\mu_k$ are called
the parts of $\mu$. Now let $\mu = (\mu_1, \ldots, \mu_r)$ be a composition of 
$n$. For each $k \in \interval{r}$ set $\alpha_k = \alpha_k(\mu) = 
\sum_{j=1}^{k-1} \mu_j + 1$ and $\beta_k = \beta_k(\mu) = \sum_{j=1}^{k} 
\mu_j$, so the interval $\interval{\alpha_k, \beta_k}$ has $\mu_k$ elements; 
we refer to this interval as the $k$-th block of $\mu$. 

Denote by $S_\mu \subset S_n$ the subgroup of bijections $\sigma$ such that
$\sigma(\interval{\alpha_k, \beta_k}) = \interval{\alpha_k, \beta_k}$ for each
$k \in \interval{r}$. This is a parabolic subgroup of $S_n$ in the sense of 
\cite{BB-coxeter-book}*{section 2.4}, and we review some of the properties
discussed there. By definition each $\sigma \in S_\mu$ is a product of 
the form $\sigma = \sigma^{(1)} \sigma^{(2)} \cdots \sigma^{(r)}$ with each 
$\sigma^{(j)}$ the identity on each $\mu$-block except the $j$-th. The length 
of an element on $S_\mu$ is $\ell(\sigma) = \ell(\sigma^{(1)}) + \cdots + 
\ell(\sigma^{(r)})$, and $S_\mu$ has a unique longest word $w_\mu$ with 
$w_\mu^{(k)}$ the longest word of the permutation group of $\interval{\alpha_k,
\beta_k}$. We set $\mu! = \# S_\mu = \mu_1! \mu_2! \cdots \mu_r!$. 

Put $\Sigma(\mu) = \{(k,i) \mid k \in \interval{r}, i \in \interval{\mu_k}
\}$ and let $\gamma_\mu: \Sigma(\mu) \to \interval{n}$ given by 
$\gamma_\mu(k,i) = i + \sum_{j=1}^{k-1} \mu_j$. This map is a bijection, and 
through it $S_\mu$ acts on $\Sigma(\mu)$. We will often identify a permutation 
$\sigma \in S_\mu$ by its action on $\Sigma(\mu)$; for example, we denote by 
$s_i^{(k)}$ the simple transposition in $S_\mu$ which acts on $\Sigma(\mu)$
by interchanging $(k,i)$ and $(k,i+1)$, leaving all other elements fixed. With
this notation, $\sigma^{(k)}$ leaves all elements of the form $(j,l)$ with 
$j \neq k$ fixed. 

Set
\begin{align*}
\sym_\mu 
  &= \frac{1}{\mu!}\sum_{\sigma \in S_\mu} \sigma;
&\asym_\mu
  &= \frac{1}{\mu!} \sum_{\sigma \in S_\mu} \sg(\sigma) \sigma.
\end{align*}
These are idempotent elements of the group algebra $\CC[S_\mu]$ and given a 
$\CC[S_\mu]$-module $V$, multiplication by $\sym_\mu$, resp. $\asym_\mu$, is 
the projection onto the symmetric, resp. antisymmetric, component of $V$.

\paragraph
\textbf{Refinements.}
\label{refinements}
A \emph{refinement} of $\mu$ is a collection of compositions $\eta = 
(\eta^{(1)},\ldots, \eta^{(r)})$ with each $\eta^{(k)}$ a composition of 
$\mu_k$. If $\eta$ is a refinement of $\mu$ then the concatenation of the 
$\eta^{(k)}$'s is also a composition of $n$, which by abuse of notation we 
will also denote by $\eta$.  

If $\eta$ refines $\mu$ then $S_\eta \subset S_\mu$. We say that $\sigma \in S_
\mu$ is a $\eta$-shuffle if it is increasing in each $\eta$-block. Among the 
elements of a coclass $\sigma S_\mu \in S_\mu/S_\eta$ there is exactly one 
$\eta$-shuffle, and this is the unique element of minimal length in the 
coclass. We denote the set of all $\eta$-shuffles in $S_\mu$ by 
$\Shuffle^\mu_\eta$. The group $S_\eta$ acts on $\Sigma(\mu)$ by restriction,
and the orbits of this action are also called the $\eta$-blocks of $\mu$.

\paragraph
\label{mu-points}
We write $\CC^{\mu} = \CC^{\mu_1} \oplus \CC^{\mu_2} \oplus \cdots \oplus
\CC^{\mu_r}$; thus $v \in \CC^{\mu}$ is an $r$-uple of vectors $(v_1, \ldots,
v_r)$ with $v_k \in \CC^{\mu_k}$. We refer to the elements of $\CC^\mu$ as 
\emph{$\mu$-points}, or simply points if the composition $\mu$ is fixed. 
For each $(k,i) \in \Sigma(\mu)$ we write 
$v_{k,i}$ for the $i$-th coordinate of $v_k$. We refer to the $v_k$'s as the 
\emph{$\mu$-blocks} of $v$, and to the $v_{k,i}$'s as the \emph{entries} of 
$v$. We say that a $\mu$-point $v$ is \emph{integral} if all its entries lie 
in $\ZZ$. Clearly $\CC^\mu$ is an affine variety, and we denote by $\CC[X_\mu]$
the polynomial algebra generated by $x_{k,i}$ with $(k,i) \in \Sigma(\mu)$, 
which is the coordinate ring of $\CC^\mu$. We also denote by $\CC(X_\mu)$ the 
field of fractions of $\CC[X_\mu]$. The group $S_\mu$ acts on $\CC^\mu$ in an 
obvious way, and this induces actions on $\CC[X_\mu]$ and $\CC(X_\mu)$.

If $\eta$ is a refinement of $\mu$ then there exists an isomorphism $\CC^\eta 
\cong \CC^\mu$, so we can talk of the $\eta$-blocks of $v \in \CC^\mu$. 
Since $\eta$ refines $\mu$ we get an inclusion $S_\eta \subset S_\mu$, and so 
$S_\eta$ acts on $\CC^\mu$ by restriction. The isomorphism $\CC^\eta \cong 
\CC^\mu$ is $S_\eta$-equivariant.

\section{Gelfand-Tsetlin modules}
\label{GT-MODULES}
For the rest of this article we fix $n \in \NN$ and set $N = 
\frac{n(n+1)}{2}$. 

\paragraph
\label{D:gt-module}
For each $k \in \interval{n}$ we denote by $U_k$ the enveloping algebra of 
$\gl(k,\CC)$, and set $U = U_n$. Inclusion of matrices in the top left corner 
induces a chain
\begin{align*}
\gl(1,\CC) \subset \gl(2, \CC) \subset \cdots \subset \gl(n,\CC),
\end{align*}
which in turn induces a chain $U_1 \subset U_2 \subset \cdots \subset U_n$. 
Denote by $Z_k$ the center of $U_k$ and by $\Gamma$ the subalgebra of $U$ 
generated by $\bigcup_{k=1}^n Z_k$. This algebra is the \emph{Gelfand-Tsetlin} 
subalgebra of $U$, and it is generated by the elements
\begin{align*}
c_{k,i}
  &= \sum_{(r_1, \ldots, r_i) \in \interval{k}^i} 
    E_{r_1, r_2} E_{r_2, r_3} \cdots E_{r_i, r_1},
  & (k,i) \in \Sigma(\mu).
\end{align*}
By work of Zhelobenko there exists an isomorphism $\iota: \Gamma \to 
\CC[X_\mu]^{S_\mu}$, given by $\iota(c_{k,i}) = \gamma_{k,i}$ where
\begin{align*}
\gamma_{k,i}
  &= \sum_{j=1}^k (x_{k,j} + k - 1)^i 
    \prod_{m \neq j} \left( 
      1 - \frac{1}{x_{k,j} - x_{k,m}}
    \right),
\end{align*}
see \cite{FGR-1-singular}*{subsection 3.1} for details. 
It follows that $\Specm \Gamma \cong \CC^\mu / S_\mu$, and so every 
$\mu$-point $v$ induces a character $\chi_v: \Gamma \to \CC$ by setting
$c \in \Gamma \mapsto \iota(c)(v)$, and two $\mu$-points induce the same 
character if and only if they lie in the same $S_\mu$-orbit. 

\begin{Definition*}
A finitely generated $U$-module $M$ is called a \emph{Gelfand-Tsetlin module} 
if
\[
  M = \bigoplus_{\m \in \Specm \Gamma} M[\mathfrak m],
\] 
where $M[\mathfrak m] = \{x \in M \mid \mathfrak m^k v = 0 \mbox{ for some } k 
\geq 0\}$.
\end{Definition*}
Let $M$ be a Gelfand-Tsetlin module.
Identifying $\mathfrak m$ with the character $\chi: \Gamma \to \Gamma / 
\mathfrak m \cong \CC$, we also write $M[\chi]$ for $M[\mathfrak m]$. We say 
that $\chi$ is a Gelfand-Tsetlin character of $M$ if $M[\chi] \neq 0$ and
define the multiplicity of $\chi$ in $M$ as $\dim_\CC M[\chi]$. The 
Gelfand-Tsetlin support of $M$ is the set of all of its Gelfand-Tsetlin 
characters. We will often abreviate Gelfand-Tsetlin for GT, or ommit it
completely when it is clear from the context.

\paragraph
\label{T:gt-tableaux}
\textbf{Gelfand-Tsetlin tableaux.}
For the rest of this document we denote by $\mu$ the composition $(1, 2, 
\ldots, n)$ of $N$. To each $\mu$-point we assign a triangular array  

\begin{tikzpicture}
\node (tv) at (-3, 1.5) {$T(v) = $};

\node (n1) at (-2,2.5) {$v_{n,1}$};
\node (n2) at (-1,2.5) {$v_{n,2}$};
\node (ndots) at (0,2.5) {$\cdots$};
\node (nn-1) at (1,2.5) {$v_{n,n-1}$};
\node (nn) at (2,2.5) {$v_{n,n}$};

\node (n-11) at (-1.5,2) {$v_{n-1,1}$};
\node (n-1dots) at (0,2) {$\cdots$};
\node (n-1n-11) at (1.5,2) {$v_{n-1,n-1}$};

\node (dots1) at (-1,1.625) {$\ddots$};
\node (dots2) at (0,1.5) {$\cdots$};
\node (dots3) at (1,1.625) {$\iddots$};

\node (21) at (-.5,1) {$v_{2,1}$};
\node (22) at (.5,1) {$v_{2,2}$};
\node (11) at (0,.5) {$v_{1,1}$};

\node (A) at (-3.5, 2.75) {};
\node (B) at (3.5, 2.75) {};
\node (C) at (0,0) {};
\end{tikzpicture}

Such an array is known as a \emph{Gelfand-Tsetlin tableau}. With this 
identification the group $S_\mu$ acts on the space of all tableaux.
The main difference between the space of $\mu$-points and the space of
all tableaux is that the first is obviously a $\CC$-vector-space, while
the second one is not. In particular $T(v+w) \neq T(v) + T(w)$, since the
second expression does not make sense (though we will eventually consider
a vector-space generated by tableaux). 

A \emph{standard} $\mu$-point is one in which $v_{k,i} - v_{k-1,i} \in 
\ZZ_{\geq 0}$ and $v_{k-1,i} - v_{k,i+1} \in \ZZ_{>0}$ for all $1 \leq i < k 
\leq n$. We denote by $\CC^\mu_\std$ the set of all standard $\mu$-points.
Notice that given $\lambda = (\lambda_1, \ldots, \lambda_n)$, there
exists finitely many standard tableaux with top row equal to $\lambda - (0, 1, 
2, \ldots, n-1)$, and that this number is nonzero if and only if $\lambda_i - 
\lambda_{i+1} \in \ZZ_{\geq 0}$, i.e. $\lambda$ must be a dominant integral 
weight of $\gl(n, \CC)$. This definition was introduced by Gelfand and Tsetlin 
in their article \cite{GT-modules} in order to give an explicit presentation 
of irreducible $\gl(n,\CC)$-modules.

\begin{Theorem*}[\cites{GT-modules, Zh-compact-book}]
Let $\lambda = (\lambda_1, \ldots, \lambda_n)$ be a dominant integral weight 
of $\gl(n,\CC)$, and let $V(\lambda)$ be the complex vector-space freely 
generated by the set
\begin{align*}
\left\{ T(v) \mid v \in \CC^\mu_{\std} 
    \mbox{ and } v_{n,1} = \lambda_1, v_{n,2} = \lambda_2 - 1, \ldots, 
    v_{n,n} = \lambda_{n} - n +1
  \right\}
\end{align*}
(by convention, if $v$ is non-standard then $T(v) = 0$ in $V(\lambda)$).
The vector-space $V(\lambda)$ can be endowed with a $\gl(n,\CC)$-module 
structure, with the action of the canonical generators given by
\begin{align*}
E_{k,k+1} T(v) 
  &= - \sum_{i=1}^k \frac{
    \prod_{j=1}^{k+1} (v_{k,i} - v_{k+1, j})}
    {\prod_{j \neq i}(v_{k,i} - v_{k,j})} T(v + \delta^{k,i}), \\
E_{k+1,k} T(v) 
  &= \sum_{i=1}^k \frac{
    \prod_{j=1}^{k-1} (v_{k,i} - v_{k-1,j})}
    {\prod_{j\neq i}(v_{k,i} - v_{k,j})} T(v - \delta^{k,i}), \\
E_{k,k} T(v)
  &= \left(\sum_{j=1}^k v_{k,j} - \sum_{j=1}^{k-1} v_{k-1,1} + k -1\right) 
    T(v),
\end{align*}
where $\delta^{k,i}$ is the element of $\ZZ^\mu$ with a $1$ in position $(k,i)$
and $0$'s elsewhere. Furthermore, for each $c \in 
\Gamma$ we have $c_{k,i} T(v) = \gamma_{k,i}(v) T(v)$. 
\end{Theorem*}
The $U$-module $V(\lambda)$ is an irreducible finite dimensional 
representation of maximal weight $\lambda$, so this theorem provides an 
explicit presentation of all finite dimensional simple $\gl(n,\CC)$-modules. 
The last statement of the theorem, giving the action of the generators of
$\Gamma$ on $V(\lambda)$, is due to Zhelobenko.

\paragraph
Given $1 \leq i \leq k < n$ we set
\begin{align*}
e_{k,i}^+ &=  \frac{
    \prod_{j=1}^{k+1} (x_{k,i} - x_{k+1, j})}
    {\prod_{j \neq i}(x_{k,i} - x_{k,j})},&
e_{k,i}^- &= \frac{
    \prod_{j=1}^{k-1} (x_{k,i} - x_{k-1,j})}
    {\prod_{j\neq i}(x_{k,i} - x_{k,j})},
\end{align*}
which are elements of $\CC(X_\mu)$. These are the rational functions that 
appear in Gelfand and Tsetlin's presentation of finite dimensional irreducible 
$U$-modules, and we refer to them as the Gelfand-Tsetlin functions. 

A $\mu$-point $v$, or equivalently the corresponding GT tableaux, is called 
\emph{generic} if $v_{k,i} - v_{k,j} \notin \ZZ$ for all $1 \leq i < j \leq k 
< n$, otherwise it is called \emph{singular}. Notice that a generic tableau 
may have entries in the top row whose difference is an integer. 

Let $\ZZ^\mu_0$ be the set of all integral $\mu$-points with $z_{n,i} = 0$ for 
all $1 \leq i \leq n$ (i.e., the corresponding GT tableau has its top row 
filled with zeros). If $v$ is generic then the set $\{v+z \mid z \in \ZZ^\mu_0
\}$ contains no poles of the Gelfand-Tsetlin functions. This idea was used by
Drozd, Futorny and Ovsienko in \cite{DFO-GT-modules} to give a $U$-module 
structure to the vector-space
\begin{align*}
V(T(v)) &= \langle T(v+z) \mid z \in \ZZ^\mu_0\rangle_\CC.
\end{align*}
Since the action of $U$ is given by the Gelfand-Tsetlin functions, each 
tableau $T(v)$ is an eigenvector of $\Gamma$ and hence $V(T(v))$ is a
GT module with support $\{\chi_{v+z} \mid z \in \ZZ^\mu_0\}$ and all 
multiplicities equal to $1$. 

\paragraph
\label{big-module}
The main goal of this article is to give a nontrivial $U$-module structure to 
the space $V(T(v))$ for an arbitrary $\mu$-point $v$. The approach from 
\cite{DFO-GT-modules} breaks down if $v$ is not generic, but Futorny, 
Grantcharov and the first named author extended this construction to a subset 
of singular characters in \cites{FGR-1-singular, FGR-2-index}, and in this
article we extend their work to arbitrary $\mu$-points. In order to do this, 
we introduce a large $U$-module which serves as a formal model of a GT module.

Set $K = \CC(X_\mu)$ to be the field of rational functions over $\mu$-points.
As mentioned above $e_{k,i}^\pm \in K$ for all $1 \leq i \leq k <n$. We set
$V_\CC$ to be the $\CC$-vector-space with basis $\left\{T(z) \mid z \in 
\ZZ^\mu_0 \right\}$, and $V_K = K \ot_\CC V_\CC$. Since $S_\mu$ acts naturally
on both $K$ and $V_\CC$, it acts on $V_K$ by the diagonal action, i.e. given 
$\sigma \in S_\mu, f \in K$ and $z \in \ZZ^\mu_0$ the action is the linear extension of $\sigma \cdot f \ot T(z) = \sigma(f) \ot T(\sigma(z))$. The group
$\ZZ^\mu_0$ also acts on $\CC[X_\mu]$, with $\delta^{k,i} \cdot x_{l,j} = 
x_{l,j} + \delta_{k,l}\delta_{i,j}$. This action extends to $K$, and given
$f \in K, z \in \ZZ^\mu_0$ we sometimes write $f(x+z)$ instead of $z \cdot f$.
The actions of $S_\mu$ and $\ZZ^\mu_0$ do not commute, since $\sigma(z\cdot f)
= \sigma(z) \cdot \sigma(f)$.

\begin{Proposition*}
The vector-space $V_K$ has an $S_\mu$-equivariant $U$-module structure, with 
the action of the generators given by
\begin{align*}
E_{k,k+1} T(z) 
  &= - \sum_{i=1}^k e^+_{k,i}(x+z) T(z + \delta^{k,i}); \\
E_{k+1,k} T(z) 
  &= \sum_{i=1}^k e^-_{k,i}(x+z) T(z - \delta^{k,i}); \\
E_{k,k} T(z)
  &=  \left(
    \sum_{j=1}^k (x_{k,j} + z_{k,j}) - 
    \sum_{j=1}^{k-1} (x_{k-1,j} + z_{k-1,j}) + k -1 \right) 
    T(z).
\end{align*}
Furthermore, for each $c \in \Gamma$ we have $c T(z) = \iota(c)(x +z) T(z)$. 
\end{Proposition*}
\begin{proof}
The proof that $V_K$ is a $U$-module is identical to the proof of
\cite{Zad-1-sing}*{Proposition 1}. It follows from the definitions that 
$\sigma \cdot e^{\pm}_{k,i} = e^\pm_{k,\sigma^{(k)}(i)}$, and using this it
is easy to check that the canonical generators of $U$ act by 
$S_\mu$-equivariant operators.
\end{proof}
We will refer to the $U$-module $V_K$ as the ``\bigmodule''. The \bigmodule 
was introduced in \cite{Zad-1-sing}, where it was proved that certain 
sublattices can serve as universal models for modules with generic or 
$1$-singular characters ($1$-singular here means that there is at most one
pair of entries in the same $\mu$-block differing by an integer). In the
following sections of this article we will extend this argument to arbitrary
characters without any restriction. 

\section{Symmetrized divided difference operators}
\label{SDD-OPERATORS}
In this section we recall some results on divided difference operators, and
introduce a symmetrized version of them. These operators will play a central 
role in the construction of the sublattices of the \bigmodule. 

\paragraph
\label{sdd-intro}
Throughout this section we fix $m \in \NN$ and $\eta = (\eta_1, \ldots, 
\eta_r)$ a composition of $m$. Recall that $\CC[X_\eta] = \CC[x_{k,i} \mid
(k,i) \in \Sigma(\eta)]$, on which $S_\eta$ acts by $\sigma(x_{k,i}) = 
x_{\sigma\cdot(k,i)}$. We set $F = \CC(X_\eta)$, the fraction field of 
$\CC[X_\eta]$, with its obvious $S_\eta$-action. 

Set 
\begin{align*}
\Delta_k 
  &= \prod_{1 \leq i < j \leq \eta_k} (x_{k,i} - x_{k,j}),
  &\Delta_\eta 
  &= \prod_{k=1}^r \Delta_k.
\end{align*}
If $f \in \CC[X_\eta]$ is a polynomial such that $\sigma(f) = \sg(\sigma) f$ 
for all $\sigma \in S_\eta$ then $f = g \Delta_\eta$, with $g$ an 
$S_\eta$-invariant polynomial. 

\paragraph
\label{ddoperators}
\textbf{Divided difference operators.} 
Since the action of $S_\eta$ extends to $F$ we can form the smash product 
$F \# S_\eta$, which is the $\CC$-algebra whose underlying vector-space 
is $F \ot \CC[S_\eta]$ and product given by $(f \ot \sigma)\cdot(g \ot \tau) 
= f\sigma(g) \ot \sigma \tau$ for all $f,g \in F$ and all $\sigma, \tau \in 
S_\eta$. We will usually ommit the tensor product symbol when writing elements 
in $F \# S_\eta$, so $f\sigma$ stands for $f \ot \sigma$. We must 
be careful to distinguish the action of $\sigma \in S_\eta$ on a rational 
function $f \in F$, denoted by $\sigma(f)$, from their product in 
$F \# S_\eta$, which is $\sigma \cdot f = \sigma(f) \sigma$.

Recall that for each $(k,i) \in \Sigma(\eta)$ we denote by $s_i^{(k)}$ the 
unique simple transposition in $S_\eta$ which acts on $\Sigma(\eta)$ by 
interchaging $(k,i)$ and $(k,i+1)$, while leaving the other elements fixed. 
The divided difference associated to $s_i^{(k)}$ is $\partial_{i}^{(k)} = 
\frac{1}{x_{k,i} - x_{k,i+1}} (\id - s_i^{(k)}) \in F \# S_\eta$. These 
elements satisfy the relations
\begin{align*}
(\partial_{i}^{(k)})^2 &= 0 \\
\partial_i^{(k)} \partial_j^{(l)} 
  &= \partial_j^{(l)}\partial_i^{(k)}
  & \mbox{ if } l \neq k \mbox{ or } |i-j|>1; \\
\partial_i^{(k)} \partial_{i+1}^{(k)} \partial_i^{(k)} 
  &=\partial_{i+1}^{(k)} \partial_i^{(k)} \partial_{i+1}^{(k)}
  & \mbox{ for } 1 \leq i \leq k -2.
\end{align*}
Let $\sigma = \sigma^{(1)} \cdots \sigma^{(n)} \in S_\eta$. For each $k \in 
\interval{n}$ we can write $\sigma^{(k)}$ as a reduced composition 
$s_{i_1}^{(k)} s_{i_2}^{(k)} \cdots s_{i_{l}}^{(k)}$, with $\ell(\sigma^{(k)}) 
= l$. We set $\partial_\sigma^{(k)} = \partial_{i_1}^{(k)} \cdot 
\partial_{i_2}^{(k)} \cdots \partial_{i_l}^{(k)}$, and 
$\partial_\sigma = \partial_\sigma^{(1)} \cdot \partial_\sigma^{(2)} 
\cdots\partial_\sigma^{(n)}$; the relations among the divided difference 
operators imply that $\partial_\sigma^{(k)}$, and hence $\partial_\sigma$, is 
independent of the reduced composition we choose for $\sigma^{(k)}$.

The equality $\ell(\sigma) + \ell(\tau) = \ell(\sigma\tau)$ holds if and only 
if the concatenation of a reduced composition of $\sigma$ with a reduced 
composition of $\tau$ is a reduced composition of $\sigma\tau$, and this 
implies that $\partial_\sigma \cdot \partial_\tau = \partial_{\sigma\tau}$.
If equality does not hold then the concatenation is not reduced, which implies
that in the product $\partial_\sigma \cdot \partial_\tau$ there must be a 
term of the form $\partial_{i}^{(k)} \cdot \partial_i^{(k)} = 0$. Thus for 
all $\sigma, \tau \in S_\eta$ we get.
\begin{align*}
\partial_\sigma \cdot \partial_\tau &= 
  \begin{cases}
  \partial_{\sigma \tau} &\mbox{ if } \ell(\sigma) + \ell(\tau) 
    = \ell(\sigma\tau); \\
  0 & \mbox{otherwise.}
  \end{cases}
\end{align*}

\paragraph
\label{L:dd-algebra}
Divided differences are usually defined as operators on the polynomial algebra 
$\CC[X_\eta]$, although they make sense over any $F$-vector-space with an
equivariant $S_\eta$-action. Since these operators play an important role in 
the sequel, we gather some of their basic properties in the following lemma.
\begin{Lemma*}
Let $w_\eta$ be the longest word in $S_\eta$. The following equalities hold in 
$F\# S_\eta$.
\begin{enumerate}
\item 
\label{i:dd-b}
$\partial_{i}^{(k)} \cdot f = s_i^{(k)}(f)\partial_i^{(k)} + \partial_i^{(k)}(f)$ for all $(k,i) \in \Sigma(\eta), f \in F$.

\item 
\label{i:dd-c}
$\partial_{w_\eta} \cdot f  \partial_\sigma = \partial_{w_\eta} \cdot 
\partial_{\sigma^{-1}}(f)$ for all $\sigma \in S_\eta, f \in F$.

\item 
\label{i:dd-d}
$\frac{1}{\eta!}\partial_{w_\eta} = \frac{1}{\Delta_\eta} \asym_\eta
= \sym_\eta \cdot \frac{1}{\Delta_\eta}$.
\end{enumerate}
\end{Lemma*} 
\begin{proof}
Item \ref{i:dd-b} is an easy computation following from the definition. Now set
$s = s_i^{(k)}$ and $\partial_s = \partial_{i}^{(k)}$. By definition 
$\partial_s \cdot s = - \partial_s$ and $s \cdot \partial_s = \partial_s$, and
since $w_\eta = (w_\eta s) s$ with $\ell(w_\eta s) = \ell(w_\eta) - 1$ we get
\begin{align*}
\partial_{w_\eta} \cdot s 
  &= \partial_{w_\eta s} \cdot \partial_s \cdot s
  = - \partial_{w_\eta s} \cdot \partial_s = - \partial_{w_\eta}.
\end{align*}
This along with the previous item gives
\begin{align*}
0 
  &= \partial_{w_\eta} \cdot (\partial_s \cdot f)
  = \partial_{w_\eta} \cdot s(f) \partial_s + \partial_{w_\eta} \cdot 
    \partial_s(f)\\
  &= \partial_{w_\eta} \cdot (s \cdot f s) \cdot \partial_s + 
    \partial_{w_\eta} \cdot \partial_s(f)
  = - \partial_{w_\eta} \cdot f \partial_s + \partial_{w_\eta} \cdot 
    \partial_s(f)
\end{align*}
which proves that item \ref{i:dd-c} holds for $\sigma = s$. Since $k$ and $i$
are arbitrary, the general case follows by induction on the length of $\sigma$.

A similar argument as above shows that $\sigma \cdot \partial_{w_\eta} = 
\partial_{w_\eta}$ for all $\sigma \in S_\eta$. Induction on the length 
of $\sigma$ shows that $\partial_\sigma = \sum_{\tau \in S_\eta} 
\frac{1}{f_{\tau,\sigma}} \tau$ with $f_{\tau,\sigma} \in \CC[X_\eta]$ 
homogeneous of degree $\ell(\sigma)$. If we put $f_\tau = f_{\tau, w_\eta}$, 
the equalities $\sigma \cdot \partial_{w_\eta} = \partial_{w_\eta}$ and 
$\partial_{w_\eta} \cdot \sigma = \sg(\sigma) \partial_{w_\eta}$ imply that 
$f_\sigma = \sigma(f_e) = \sg(\sigma) f_e$, so $f_e$ is an 
$S_\eta$-antisymmetric polynomial of degree $\ell(w_\eta)$, i.e. a scalar 
multiple of $\Delta_\eta$. Now $\partial_{w_\eta}(\Delta_\eta) = \eta!$, so 
$f_e = \frac{\eta!}{\Delta_\eta}$ and
\begin{align*}
\partial_{w_\eta} 
  &= \frac{\eta! }{\Delta_\eta} \sum_{\sigma \in S_\eta} \sg(\sigma) \sigma
  = \eta! \sum_{\sigma \in S_\eta} \sigma \cdot \frac{1}{\Delta_\eta}.
\end{align*}
This completes the proof of item \ref{i:dd-d}.
\end{proof}
Many subalgebras of $F$ are stable by the action of divided differences.
It is a well-known fact that this is the case for $\CC[X_\eta]$. Assume 
$\CC[X_\eta] \subset A \subset F$ is closed under the action of $S_\eta$. 
Then any rational function in $A$ can be written as a quotient $p/q$ with 
$p,q \in \CC[X_\eta]$ and $q$ $S_\eta$-invariant, so $\partial_\sigma(f) = 
\partial_\sigma(p/q) = \partial_\sigma(p)/q \in A$ for each $\sigma \in 
S_\eta$ and $A$ is also closed under divided differences.

\paragraph
\label{polynomial-dd}
We focus now on the action of divided differences over the polynomial ring 
$\CC[X_\eta]$. First, for every $\sigma \in S_\eta$ the operator
$\partial_\sigma: \CC[X_\eta] \to \CC[X_\eta]$ is homogeneous of degree 
$-\ell(\sigma)$. Now let $\p_\eta$ be the ideal generated by $\{x_{k,i} 
- x_{k,j} \mid (k,i),(k,j) \in \Sigma(\eta)\}$, and let $\CC[\p_\eta]
\subset \CC[X_\eta]$ be the algebra generated by $\p_\eta$; clearly 
$\Delta_\eta
\in \p_\eta$. Since $\partial_i^{(k)} (x_{l,j} - x_{l,t}) \in \ZZ$, the 
algebra $\CC[\p_\eta]$ is stable by the action of divided differences. It 
follows that $\partial_\sigma \Delta_\eta \in \p_\eta$ if $\ell(\sigma) <
\deg \Delta_\eta = \eta!$, while $\partial_{w_\eta} \Delta_\eta = \eta!$.

Let $\sigma, \tau \in S_\eta$. By item \ref{i:dd-c} of Lemma 
\ref{L:dd-algebra} we have 
\begin{align*}
\frac{1}{\eta!}\partial_{w_\eta}
  ((\partial_{\tau} \Delta_\eta) (\partial_\sigma \Delta_\eta))
  &= \frac{1}{\eta!}
    \partial_{w_\eta}(\Delta (\partial_{\tau^{-1}} \partial_\sigma \Delta_\eta))
  = \sym_\eta(\partial_{\tau^{-1}} \partial_\sigma \Delta_\eta).
\end{align*}
Now if $\ell(\tau^{-1}) + \ell(\sigma) \geq \ell(w_\eta)$ this is $0$ unless
$\tau^{-1}\sigma = w_\eta$, in which case it equals $\eta!$. If 
$\ell(\tau^{-1}) + \ell(\sigma) < \ell(w_\eta)$ then the best we can say is 
that this element lies in $\p_\eta^{S_\eta}$. 

Consider a total order $<$ on $S_\eta$ such that $\ell(\sigma) < \ell(\tau)$
implies $\sigma < \tau$, and use this to index the rows and columns of 
matrices in $M_{\eta!}(F)$ by elements in $S_\eta$. For each 
$X \in M_{\eta!}(F)$ write $X^\sigma_\tau$ for the entry in the $\sigma$-th 
row and the $\tau$-th column, so if $Y$ is another matrix then 
$(XY)^\sigma_\tau = \sum_{\rho \in S_\eta} X^\sigma_\rho Y^\rho_\tau$.
Consider the matrices $X,Y \in M_{\eta!}(F)$ defined by $X^\sigma_\tau 
= \tau \left( \frac{\partial_\sigma \Delta_\eta}{\eta !\Delta_\eta} \right)$ 
and $Y^\rho_\nu = \frac{1}{\eta!}\rho(\partial_{\nu w_\eta} \Delta_\eta)$. 
Then 
\begin{align*}
(XY)^\sigma_\nu
  &= \frac{1}{\eta!^2} \sum_{\tau \in S_\eta} 
  \tau \left(
    \frac{(\partial_\sigma \Delta_\eta)(\partial_{\nu w_\eta} \Delta_\eta)}
    {\Delta_\eta}
  \right) 
  = \frac{1}{\eta!^2} \partial_{w_\eta} ((\partial_\sigma \Delta_\eta) 
    (\partial_{\nu w_\eta} \Delta_\eta)).
\end{align*}
By the previous discussion $XY$ is an upper triangular matrix with ones in
the diagonal, and the nonzero elements in the upper-triangular part lie in 
$\p^{S_\eta}$. From this we deduce that $X$ is invertible, and its inverse is 
of the form $Y + C$ with $C$ a matrix with entries in the ideal generated by 
$\p^{S_\eta}$. 

The rational function $X^\sigma_\tau$ is homogeneous of degree $-\ell(\sigma)$,
hence its determinant is also homogeneous of degree $-\sum_{\sigma \in S_\eta}
\ell(\sigma) = -\eta!$. This implies that $\deg (Y+C)^\rho_\nu = \eta! + 
\ell(\nu) - \eta! = \ell(\nu) = \deg Y^\rho_\nu$, so the entries of $C$ are
also homogeneous polynomials. Also since $X^\sigma_\tau = \tau(X_{\id}^\sigma)$
and $Y^{\rho}_\nu = \rho(Y^{\id}_\nu)$, we must have $C^\rho_\nu = 
\rho(C^{\id}_\nu)$. We summarize these results in the following lemma. 
\begin{Lemma*}
For each $\sigma \in S_\eta$ set $(\partial_\sigma \Delta)^* = 
(Y + C)^\id_{\sigma^{-1}}$, and let $I$ be the ideal generated by 
$\p_\eta^{S_\eta}$ in $\CC[X_\eta]$. Then $(\partial_\sigma \Delta)^*$
is a homogeneous polynomial of degree $\ell(\sigma)$, and
$(\partial_\sigma \Delta)^* \equiv \frac{1}{\eta!} 
\partial_{\sigma^{-1}w_\eta} \Delta_\eta \mod I$. Furthermore
$(Y + C)^\tau_{\sigma^{-1}} = \tau((\partial_\sigma \Delta)^*)$
for each $\tau \in S_\eta$.
\end{Lemma*}

\paragraph
\label{P:sdd}
There is a second family of elements in $F \# S_\eta$ that we need to 
distinguish before going on. For each $\sigma \in S_\eta$ we define the
\emph{symmetrized divided difference operator} $D_\sigma^\eta = \sym_\eta 
\cdot \partial_\sigma$. It follows from the discussion at the end of
\ref{L:dd-algebra} that if $A \subset F$ is a $\CC[X_\eta]$-algebra stable by 
the action of $S_\eta$ then $D_\sigma^\eta(A) \subset A$.

\begin{Proposition*}
For each $\sigma \in S_\eta$ we have
\begin{align*}
D_\sigma^\eta 
  &= D_{w_\eta}^\eta \cdot \partial_{\sigma^{-1}}\Delta_\eta
  = \frac{1}{\eta!}\sum_{\tau \in S_\eta} \tau \left(
    \frac{\partial_{\sigma^{-1}}\Delta_\eta}{\Delta_\eta} 
    \right) \tau,\\
\sigma
  &= \sum_{\tau \in S_\eta} 
    \sigma((\partial_\tau \Delta_\eta)^*)
    D_\tau^\eta. 
\end{align*}  
\end{Proposition*}
\begin{proof}
By item \ref{i:dd-d} of Lemma \ref{L:dd-algebra} $D^\eta_{w_\eta} = 
\frac{1}{\eta!} \partial_{w_\mu} \cdot \Delta_\eta $, so by item \ref{i:dd-c}
of the same lemma
\begin{align*}
D_\sigma^\eta 
  &= \frac{1}{\eta!} \partial_{w_\eta} \cdot \Delta_\eta \partial_\sigma
  = \frac{1}{\eta!} \partial_{w_\eta} \cdot \partial_{\sigma^{-1}} \Delta_\eta
  = \sym_\eta \cdot \frac{\partial_{\sigma}\Delta_\eta}{\Delta_\eta}.
\end{align*}
Using the fact that $\tau \cdot f = \tau(f) \tau$ the first equality follows.
In terms of the matrices $X, Y$ from the previous paragraph, this says that
$D_\sigma^\eta = \sum_{\tau \in S_\eta} X_\tau^{\sigma^{-1}} \tau$. Since $X$
is invertible with inverse $Y + C$, we get that
\begin{align*}
\sigma &= \sum_{\tau \in S_\eta} (Y+C)_{\tau^{-1}}^\sigma D_\tau^\eta
\end{align*}
which is the second equality.
\end{proof}

As a nice application of this proposition notice that for each $f \in F$ 
\begin{align*}
f &= \sum_{\sigma \in S_\eta} D_\sigma^\eta (f) 
  (\partial_{\sigma} \Delta_\eta)^*.
\end{align*}
In particular, if $A$ is a $\CC[X_\mu]$-subalgebra of $F$ and it is stable by 
the action of $S_\eta$ then the set $\{(\partial_\sigma \Delta_\eta)^* \mid 
\sigma \in S_\eta\}$ is a basis of $A$ over $A^{S_\eta}$, and symmetrized 
divided differences give the coefficients of each element in this basis.

\section{Singular GT modules}
\label{SINGULAR-GT}
Recall that we have fixed $n \in \NN$ and $U = U(\gl(n, \CC))$. Also, we set 
$N = \frac{n(n+1)}{2}$ and we put $\mu = (1,2,\ldots, n)$ and we 
denote by $\ZZ^\mu_0$ the set of all integral $\mu$-points with $z_{n,i} = 0$ 
for all $1 \leq i \leq n$. Recall that if $\eta$ is a refinement of $\mu$ then 
every $\mu$-point can be seen as an $\eta$-point, and we may speak freely of 
the $\eta$-blocks of a $\mu$-point. Finally, recall $K = \CC(X_\mu)$

\paragraph
\label{singularity}
To each $\mu$-point $v$, or to its corresponding tableau $T(v)$, we associate 
a refinement of $\mu$ which we denote by $\eta(v)$ which will act as a 
measure of how far is a tableau from being generic. For any $k \in 
\interval{n-1}$ form a graph with vertices $\interval{k}$, and put an edge 
between $i$ and $j$ if and only if $v_{k,i} - v_{k,j} \in \ZZ$; the resulting 
graph is the disjoint union of complete graphs, and we set $\eta^{(k)}$ to be 
the cardinalities of each connected component arranged in descending order. 
Finally we set $\eta(v) = (\eta^{(1)}, \ldots, \eta^{(n-1)}, 1^n)$, where $1^n$
denotes the composition of $n$ consisting of $n$ ones. Thus if $v$ is generic 
then $\eta^{(k)}(v) = 1^k$, and if it is a singular tableaux then $\eta(v)$ 
will have at least one part larger than $1$. 
\begin{Definition*}
Given $v \in \CC^\mu$ the composition $\eta(v)$ is called the 
\emph{singularity} of $v$. Given $\theta$ a refinement of $\mu$, we will say 
that $v$ is \emph{$\theta$-singular} if $\eta(v) = \theta$.
\end{Definition*}

Recall that characters of the GT subalgebra $\Gamma$ are in one to one 
correspondence with $S_\mu$-orbits of $\CC^\mu$. Now if $v \in \CC^\mu$ and 
$\sigma \in S_\mu$ then $\eta(v) = \eta(\sigma(v))$, so the singularity of
a character is well-defined. We say that $v$ is in \emph{normal form}
if $v_{k,i} - v_{k,j} \in \ZZ_{\geq 0}$ implies that $v_{k,i}$ and $v_{k,j}$ 
lie in the same $\eta(v)$-block of $v$ and that $i>j$. 

We denote by $\st(v)$ the stabilizer of $v$ in $S_\eta$. We say that $v$ is 
\emph{fully critical} if it is in normal form and $\st(v) = S_\eta$, or in 
other words if $v_{k,i} - v_{k,j} \in \ZZ$ implies $v_{k,i} = v_{k,j}$. We 
will say that $v$ is fully $\eta$-critical if it is both $\eta$-singular and 
fully critical. By definition every character has at least one representative 
in normal form (though it may have many), and if $v \in \CC^\mu$ is in normal 
form then there exists $z \in \ZZ^\mu_0$ such that $v+z$ is fully critical.

\begin{Example*}
Suppose $n = 5$ and let $v\in \CC^\mu$ be the point whose corresponding
tableau is

\begin{tikzpicture}
\node (tv) at (-3, 1.5) {$T(v) = $};

\node (51) at (-2,2.5) {$*$};
\node (52) at (-1,2.5) {$*$};
\node (53) at (0,2.5) {$*$};
\node (54) at (1,2.5) {$*$};
\node (55) at (2,2.5) {$*$};

\node (41) at (-1.5,2) {$a$};
\node (42) at (-0.5,2) {$b-1$};
\node (43) at (0.5,2) {$b$};
\node (44) at (1.5,2) {$a+1$};

\node (31) at (-1,1.5) {$c$};
\node (32) at (0,1.5) {$c+1$};
\node (33) at (1,1.5) {$d$};

\node (21) at (-.5,1) {$e$};
\node (22) at (.5,1) {$e$};

\node (11) at (0,0.5) {$f$};

\node (A) at (-3, 2.75) {};
\node (B) at (3, 2.75) {};
\node (C) at (0,0) {};

\end{tikzpicture}

\noindent where $a,b,c,d,e,f \in \CC$ are $\ZZ$-linearly independent. 
Its singularity is given by the refinement $\eta(v) = 
((1),(2),(2,1),(2,2),(1,1,1,1,1))$. It is not in normal form, since for 
example in the third row $c$ is to the left of $c+1$, and in the fourth row
the entries differing by integers are not organazied in $\eta(v)$-blocks.
The $\mu$-points $v', v''$ whose tableaux are 

\begin{tikzpicture}
\node (tv) at (-3, 1.5) {$T(v') = $};

\node (51) at (-2,2.5) {$*$};
\node (52) at (-1,2.5) {$*$};
\node (53) at (0,2.5) {$*$};
\node (54) at (1,2.5) {$*$};
\node (55) at (2,2.5) {$*$};

\node (41) at (-1.5,2) {$a+1$};
\node (42) at (-0.5,2) {$a$};
\node (43) at (0.5,2) {$b$};
\node (44) at (1.5,2) {$b-1$};

\node (31) at (-1,1.5) {$c+1$};
\node (32) at (0,1.5) {$c$};
\node (33) at (1,1.5) {$d$};

\node (21) at (-.5,1) {$e$};
\node (22) at (.5,1) {$e$};

\node (11) at (0,0.5) {$f$};

\node (A) at (-3, 2.75) {};
\node (B) at (3, 2.75) {};
\node (C) at (0,0) {};

\end{tikzpicture}
\begin{tikzpicture}
\node (tv) at (-3, 1.5) {$T(v'') = $};

\node (51) at (-2,2.5) {$*$};
\node (52) at (-1,2.5) {$*$};
\node (53) at (0,2.5) {$*$};
\node (54) at (1,2.5) {$*$};
\node (55) at (2,2.5) {$*$};

\node (41) at (-1.5,2) {$b$};
\node (42) at (-0.5,2) {$b-1$};
\node (43) at (0.5,2) {$a+1$};
\node (44) at (1.5,2) {$a$};

\node (31) at (-1,1.5) {$c+1$};
\node (32) at (0,1.5) {$c$};
\node (33) at (1,1.5) {$d$};

\node (21) at (-.5,1) {$e$};
\node (22) at (.5,1) {$e$};

\node (11) at (0,0.5) {$f$};

\node (A) at (-3, 2.75) {};
\node (B) at (3, 2.75) {};
\node (C) at (0,0) {};

\end{tikzpicture}

\noindent are in normal form, and since both are in the $S_\mu$-orbit of $v$
they define the same character of $\Gamma$. None of these $\mu$-points is 
fully critical, but it is clear that we may obtain a fully critical 
$\mu$-point by adding a suitable $z \in \ZZ^\mu_0$ to either $v'$ or $v''$. 
Notice that in all these cases the entries in the top row are irrelevant.
\end{Example*}

\paragraph
\label{L:derived-tableaux}
Let $C = \{x_{k,i} - x_{k,j} - z \mid 1 \leq i < j \leq k < n, z \in 
\ZZ\setminus \{0\}\} \subset \CC[X_\mu]$. We put $B = C^{-1} \CC[X_\mu]$.
Let $v \in \CC^\mu$ be fully critical, and fix $\eta = \eta(v)$. We denote by 
$B_\eta$ the localization of $B$ at the set of all $x_{k,i} - x_{k,j}$ such 
that $(k,i)$ and $(k,j)$ lie in different orbits of $S_\eta$. 
The algebra $B_\eta$ is closed under the action of $S_\eta$, and hence by the 
discussion at the end of paragraph \ref{L:dd-algebra}, if $f \in B_\eta$ then 
$D_\sigma^\eta(f) \in B_\eta$.

\begin{Definition*}
We set $L_\eta \subset V_K$ to be the $B_\eta$-span of $\{D_\sigma^\eta T(z) 
\mid \sigma \in S_\eta, z \in \ZZ^\mu_0\}$. 
\end{Definition*}

We will now prove that $L_\eta$ is a $U$-submdule of $V_K$, but in order to
do this we first need to show that it is a free $B_\eta$-module.

Notice that if $z \in \ZZ^\mu_0$ then $\mu$-point $v+z$ is in normal form if 
and only if each $\eta$-block of $z$ is a descending sequence, so the set
$\N_\eta = \{z \in \ZZ^\mu_0 \mid v+z \mbox{ is in normal form }\}$.
depends only of $\eta$ and not of $v$. Also, there is exactly one 
element in the orbit of $z$ which lies in $\N_\eta$, so this set is a family
of representatives of $\ZZ^\mu_0 / S_\eta$. We denote by $\st(z)$ the 
stabilizer of $z$ in $S_\eta$. If $z \in \N_\eta$ then $\st(z) = 
S_{\epsilon(z)}$ with $\epsilon(z)$ a refinement of $\eta$.  
\begin{Lemma*}
The set $\mathcal B =\left\{D_\sigma^\eta T(z) \mid z \in \N_\eta, \sigma \in 
\Shuffle_{\epsilon(z)}^\eta\right\}$ is a basis of $L_\eta$ as $B_\eta$-module.
\end{Lemma*}
\begin{proof}
For each $z \in \ZZ^\mu_0$ we denote by $\O(z)$ the $K$-vector-space generated 
by $\{T(\sigma(z)) \mid \sigma \in S_\eta\}$, and by $\DD(z)$ the 
$B_\eta$-module generated by $\{D_\sigma^\eta T(z) \mid \sigma \in S_\eta \}$. 
Clearly $L_\eta = \sum_{z \in \ZZ_0^\mu} \DD(z)$. By Proposition \ref{P:sdd}
\begin{align*}
D_\sigma^\eta T(z)
  &= \frac{1}{\eta!}\sum_{\tau \in S_\eta} \tau \left(
    \frac{\partial_{\sigma^{-1}}\Delta_\eta}{\Delta_\eta} 
    \right) T(\tau(z)) \\
T(\sigma (z)) &= \sum_{\tau \in S_\eta} \sigma((\partial_{\tau} \Delta_\eta)^*)
  D_\tau^\eta T(z),
\end{align*}
so $\DD(z) \subset \O(z)$ and $\O(z) \subset K \ot_{B_\eta} \DD(z)$. 

Let $v \in V_K$ be $S_\eta$-invariant and let $s \in S_\eta$ be a simple 
transposition. Then for each $f \in K$ we have $\partial_s(f v) = 
\partial_s(f) s(v) +  f \partial_s(v) = \partial_s(f) v$. A simple induction 
shows that $\partial_\sigma(fv) = \partial_\sigma(f) v$, and hence 
$D_\sigma^\eta(f v) = D_\sigma^\eta(f) v$ for all $\sigma \in S_\eta$.
Thus for each $\nu \in S_\eta$
\begin{align*}
D_\nu^\eta T(\sigma (z))
  &= D_\nu^\eta \left(
    \sum_{\tau \in S_\eta} \sigma((\partial_{\tau} \Delta_\eta)^*)
  D_\tau^\eta T(z)
  \right)
  = \sum_{\tau \in S_\eta} D_\nu^\eta(\sigma(\partial_{\tau} \Delta_\eta)^* )
  D_\tau^\eta T(z),
\end{align*}
which shows that $\DD(z) = \DD(\tau(z))$, and hence $L_\eta = \sum_{z \in 
\N_\eta} \DD(z)$. Since $\DD(z) \subset \O(z)$, the sum is direct and 
$L_\eta = \bigoplus_{z \in \N_\eta} \DD(z)$. Hence to prove the statement it 
is enough to show that for each $z \in \N_\eta$ the $B_\eta$-module $\DD(z)$ 
is free with basis $\mathcal B(z) = \mathcal B \cap \DD(z)$.

Let $z \in \N_\eta$ and put $\epsilon = \epsilon(z)$. If $\sigma \in S_\eta$ 
is not an $\epsilon$-shuffle then $\sigma$ can be written as $\tilde \sigma s$ 
with $s \in S_{\epsilon}$ a simple transposition. Since $s T(z) = T(z)$
we get
\begin{align*}
D_\sigma^\eta T(z) 
  &= \sym_\eta (\partial_{\tilde \sigma} (\partial_s T(z))) = 0.
\end{align*}
This implies that $\mathcal B(z)$ generates $\DD(z)$ over $B_\eta$, and
since $\O(z) = K \ot_{B_\eta} \DD(z)$ it also generates $\O(z)$ over $K$.
The set $\Shuffle_{\epsilon}^\eta$ is a complete set of representatives of 
$S_\eta / S_{\epsilon}$, so $\# \mathcal B(z) = \# \Shuffle_{\epsilon}^\eta = 
\eta!/\epsilon!$. On the other hand $\dim_K \O(z) = \# (S_\eta / S_\epsilon) = 
\eta!/\epsilon!$, so $\mathcal B(z)$ is a basis of $\O(z)$, and in particular
it is linearly independent over $B_\eta$. 
\end{proof}
We refer to the elements of $\mathcal B$ as \emph{derived tableaux}, and the
elements in $\mathcal B(z)$ as the derived tableaux of $T(z)$.

\begin{Theorem}
The $B_\eta$-lattice $L_\eta$ is a $U$-submodule of $V_K$.
\end{Theorem}
\begin{proof}
Let $z \in \N_\eta$ and let $\epsilon$ be the unique refinement of $\eta$ such 
that $S_\epsilon \subset S_\eta$ is the stabilizer of $z$ in $S_\eta$. For 
each $k \in \interval{n}$ the $K$-vector-space $\O(z)$ consists of 
eigenvectors of $E_{k,k}$ with the same eigenvalue; since every derived 
tableaux of $T(z)$ lies in $\O(z)$, we get that $E_{k,k} L_\eta \subset 
L_\eta$. Hence we only need to show that $E L_\eta \subset L_\eta$ for $E \in 
\{E_{k,k+1}, E_{k+1,k} \mid 1 \leq k \leq n-1\}$

It follows from the definitions that $\Delta_\epsilon e_{k,i}^\pm(x+z) \in 
B_\eta$ for each $k \in \interval{n-1}$, so $\Delta_\epsilon E T(z)$ is a 
linear combination of tableaux $T(w)$. Now since $T(w) \in L_\eta$ for all $w 
\in \ZZ^\mu_0$ and $\mathcal B$ is a basis of $L_\eta$, we can write
\begin{align*}
E T(z)
  &= \sum_{D_\sigma^\eta T(w) \in \mathcal B} 
    \frac{f_{\sigma,w}}{\Delta_\epsilon} D_\sigma^\eta T(w)
\end{align*}
with $f_{\sigma,w} \in B_\eta$ unique. 

Since the action of $U$ is $S_\eta$-equivariant and $T(z)$ is stable by
$S_\epsilon$, the same is true for $E T(z)$, and hence the right hand side of
the equation is also $S_\epsilon$-invariant. Since derived tableaux are 
$S_\eta$-invariant we see that $\frac{f_{\sigma,w}}{\Delta_\epsilon}$ is an
$S_\epsilon$-invariant element of $K$, and hence $\tau(f_{\sigma,w}) = 
\sg(\tau) f_{\sigma,w}$ for each $\tau \in S_\epsilon$. This implies that
$f_{\sigma,w} = g_{\sigma,w} \Delta_\epsilon$ with $g_{\sigma,w} \in 
B_\eta^{S_\epsilon}$, and hence $E T(z) \in L_\eta$. Finally, since the action
of $U$ is both $K$-linear and $S_\eta$-equivariant we obtain 
\begin{align*}
E D_\nu^\eta T(z)
  &= D_\nu^\eta (ET(z))
  = \sum_{D_\sigma^\eta T(w) \in \mathcal B} 
    D_\nu^\eta(g_{\sigma, w}) D_\sigma^\eta T(w).
\end{align*}
Now $B_\eta$ is stable by the action of $S_\eta$ and closed under divided 
differences, so $D_\nu^\eta(g_{\sigma, w}) \in B_\eta$ and $E D_\nu^\eta 
T(z) \in L_\eta$.
\end{proof}

\paragraph
\label{L:pre-gt}
Let $v \in \CC^\mu$ be an $\eta$-critical point. Then there is a well-defined
map $\pi_v: B_\eta \to \CC$ given by $\pi_v(f) = f(v)$, and it is clear by 
definition that $\p_\eta \subset \ker \pi_v$. From this we obtain a 
one-dimensional representation of $B_\eta$, which we denote by $\CC_v$. We fix 
a nonzero element $1_v \in \CC_v$.
\begin{Definition*}
Let $v \in \CC^\mu$ be an $\eta$-critical point. We define $V(T(v))$ to be
the complex vector-space $\CC_v \ot_{B_\eta} L_\eta$, with the $U$-module
structure given by the action of $U$ on $L_\eta$.
\end{Definition*}
Given $z \in \N$ and $\sigma \in \Shuffle_{\epsilon(z)}^\eta$, we write
$\D_\sigma^\eta(v+z) = 1_v \ot D_\sigma^\eta(z)$. It follows from Lemma 
\ref{L:derived-tableaux} that the set $\{\D_\sigma^\eta(v+z) \mid z \in 
\N_\eta, \sigma \in \Shuffle_{\epsilon(z)}^\eta\}$ is a basis of $V(T(v))$
as a complex vector-space. 
\begin{Example*}
Fix $n = 4$, let $\eta = (1,1^2, 3, 1^4)$ and let $v$ be any $\eta$-singular 
$\mu$-point. The following table shows the nonzero derived tableaux of 
$T(v+z)$, classified according to the composition $\epsilon(z)$; we always 
assume $a > b > c$.

\begin{tabular}{|c|c|l|}
\hline
$z_3$ & $\epsilon(z)^{(3)}$ & Nonzero derived tableaux \\
\hline
$(a,b,c)$
  & $(1,1,1)$
  & \parbox[c]{9cm}
    {$\D^\eta_\id(v+z), \D^\eta_{(12)}(v+z), \D^\eta_{(23)}(v+z),$\\ $\D^\eta_{(123)}(v+z), 
    \D^\eta_{(132)}(v+z), \D^\eta_{(13)}(v+z)$} \\
\hline
$(a,a,b)$
  & $(2,1)$
  & $\D^\eta_\id(v+z), \D^\eta_{(23)}(v+z), \D^\eta_{(123)}(v+z)$ \\
\hline
$(a,b,b)$
  & $(1,2)$ 
  & $\D^\eta_\id(v+z), \D^\eta_{(12)}(v+z), \D^\eta_{(132)}(v+z)$ \\
\hline
$(a,a,a)$
  & $(3)$
  & $\D^\eta_\id(v+z)$ \\
\hline
\end{tabular}
\end{Example*}

\paragraph
We finish with the proof that $V(T(v))$ is always a Gelfand-Tsetlin module, 
and we find its support along with the multiplicity of each character. 
We need the following preeliminary result.

\begin{Lemma*}
Let $\sigma, \nu, \tau \in S_\eta$. Then $d_{\sigma, \tau}^\nu = D_{\nu}^\eta(
(\partial_{\sigma} \Delta_\eta)^* (\partial_{\tau} \Delta_\eta)^* ) \equiv 0
\mod \p_\eta$ unless $\ell(\sigma) + \ell(\tau) = \ell(\nu)$. Furthermore
$d_{\nu, \id}^\nu = d_{\id, \nu}^\nu = 1$.
\end{Lemma*}
\begin{proof}
Recall that $(\partial_\sigma \Delta_\eta)^* = \frac{1}{\eta!}
\partial_{\sigma^{-1} w_\mu} \Delta_\eta + c_{\sigma^{-1}w_\eta}$ with 
$c_{\sigma^{-1}w_\eta}$ lying in the ideal generated by $\p_\eta^{S_\eta}$, 
and that this is a homogeneous polynomial of degree $\ell(\sigma)$.
Recall also that the algebra $\CC[\p_\eta]$ generated by $\p_\eta$ is closed 
by the action of the symmetrized divided difference operators, and that
a polynomial in $\CC[\p_\eta]$ lies in $\p_\eta$ if and only if it is of 
strictly positive degree.

By definition $\deg (\partial_\sigma \Delta_\eta)^* (\partial_\tau 
\Delta_\eta)^* = \ell(\sigma) + \ell(\tau)$. This implies that if 
$\ell(\nu) > \ell(\sigma) + \ell(\tau)$ then $d_{\sigma, \tau}^\nu = 0$,
so we may assume from now on that $\ell(\nu) \leq \ell(\sigma) + \ell(\tau)$.
Now if $f \in \p_\eta^{S_\eta}$ then $D_\nu^\eta(fg) = f D_\nu^\eta(g)$ for
any $g \in \CC[X_\mu]$, so 
\begin{align*}
d^\nu_{\sigma, \tau} 
  &\equiv \frac{1}{\eta!^2} D_\nu^\eta((\partial_{\sigma^{-1} w_\mu} 
    \Delta_\eta) (\partial_{\tau^{-1} w_\mu} \Delta_\eta)) \mod \p_\eta.
\end{align*} 
The polynomial in the right hand side of this congruence lies in 
$\CC[\p_\eta]$, and if $\ell(\nu) < \ell(\tau) + \ell(\sigma)$ then its degree
is positive and hence it lies in $\p_\eta$. On the other hand
\begin{align*}
d^\nu_{\nu, \id} 
  &= d^\nu_{\id, \nu}
  = \frac{1}{\eta!^2} \sym_\eta \partial_\nu((\partial_{\nu^{-1} w_\eta} 
    \Delta_\eta) (\partial_{w_\mu} \Delta_\eta)) 
  = \frac{1}{\eta!^2} \sym_\eta (\partial_{w_\eta} \Delta_\eta)^2
  = 1.
\end{align*} 
\end{proof}

\begin{Theorem}
Let $v \in \CC^\mu$ be a critical $\mu$-point, and set $\eta = \eta(v)$. The 
$U$-module $V(T(v))$ is a Gelfand-Tsetlin module whose support is the set 
$\{\chi_{v+z} \mid z \in \N_\eta\}$. Furthermore the multiplicity of 
$\chi_{v+z}$ in $V(T(v))$ is $\eta!/\epsilon(z)!$.
\end{Theorem}
\begin{proof}
Fix $c \in \Gamma$ and let $\gamma = \iota(c)$. Since $c T(z) = \gamma(x+z) 
T(z)$ for each $z \in \ZZ^\mu_0$, we see that
\begin{align*}
c D_\nu^\eta T(z)
  &= D_\nu^\eta( \gamma(x+z) T(z)).
\end{align*}
Now using Proposition \ref{P:sdd}
\begin{align*}
\gamma(x+z) 
  &= \sum_{\sigma \in S_\eta} (\partial_{\sigma} \Delta_\eta)^*
    D_\sigma(\gamma(x+z)) , \\
T(z)
  &= \sum_{\tau \in S_\eta} (\partial_{\tau} \Delta_\eta)^* 
    D_\tau T(z),
\end{align*}
and plugging this in $D_\nu^\eta( \gamma(x+z) T(z))$ we get
\begin{align*}
c D_\nu^\eta T(z)
  &= \sum_{\sigma,\tau \in S_\eta} D_{\nu}^\eta(
  (\partial_{\sigma} \Delta_\eta)^* (\partial_{\tau} \Delta_\eta)^* )
  D_\sigma^\eta(\gamma(x+z)) D_\tau^\eta T(z).
\end{align*}

By Lemma \ref{L:pre-gt} we get
\begin{align*}
c D_\nu^\eta T(z)
  &\equiv \sym_\eta(\gamma(x+z)) D_\nu^\eta T(z) \\
  &+ \sum_{\substack{\ell(\sigma) + \ell(\tau) = \ell(\nu) \\ \ell(\sigma)>0}} 
    d^\nu_{\sigma, \tau} D_\sigma^\eta(\gamma(x+z)) D_\tau^\eta T(z) 
        \mod \p_\eta L_\eta.
\end{align*}
and using the fact that $\p_\eta \subset \ker \pi_v$ and that $\pi_v(\sym_\eta
f) = f(v)$ for each $f \in \CC[X_\eta]$, 
\begin{align*}
c \D_\nu^\eta (v+z)
  &= \gamma(v+z) \D_\nu^\eta (v+z) \\
  &+ \sum_{\substack{\ell(\sigma) + \ell(\tau) = \ell(\nu) \\ 
      \ell(\tau)<\ell(\nu)}} 
    d^\nu_{\sigma, \tau} \pi_v(D_\sigma^\eta(\gamma(x+z))) \D_\tau^\eta (v+z).
\end{align*}
This shows that $(c-\gamma(v+z))^{\ell(\nu)} \D_\nu^\eta(v+z) = 0$ and hence
$\D_\nu^\eta(v+z) \in V(T(v))[\chi_{v+z}]$. It follows that
\begin{align*}
V(T(v))[\chi_{v+z}]
  &= \langle \D_\nu^\eta(v+z) \mid \nu \in \Shuffle_{\epsilon(z)}^\eta
    \rangle_\CC
\end{align*}
so the multiplicity of $\chi_{v+z}$ in $V(T(v))$ is $\eta!/\epsilon(z)!$.
\end{proof}

\begin{Remark}
Let $\theta$ be a refinement of $(1,2, \ldots, n-1, 1^n)$.
Let $v \in \CC^\mu$ be fully critical and suppose $\eta = \eta(v)$ is
a refinement of $\theta$. Then $B_\theta \subset B_\eta$ and $\CC_v$ is a 
$B_\theta$-module by restriction, so we get a $U$-module by setting $W(T(v))
= \CC_v \ot_{B_\theta} L_\theta$. It is natural to ask whether $W(T(v))$ is 
equal to $V(T(v))$. 

The answer to this question is yes. Since $S_\eta \subset S_\theta$ the 
elements $D^\theta_\tau T(z)$ are $S_\eta$-invariant for each $\tau \in 
S_\theta$, and hence
\begin{align*}
D_\sigma^\eta T(z)
  &= \sum_{\substack{\tau \in S_\theta\\ \ell(\tau) \geq \ell(\sigma)}}
    D_\sigma^\eta ((\partial_\tau \Delta_\theta)^*) D_\tau^\theta T(z)
    \in L_\theta.
\end{align*} 
This implies that $L_\eta \subset B_\eta \ot_{B_\theta} L_\theta$.

Now if $f \in B_\eta$ then there exist $f_\sigma \in B_\eta^{S_\eta}$
such that $f = \sum_{\sigma \in S_\eta} f_\sigma \partial_{\sigma^{-1}}
\Delta_\eta$, and so
\begin{align*}
\sym_\eta \left( \frac{f}{\Delta_\eta} T(z)\right)
  &= \sum_{\sigma \in S_\eta} f_\sigma \sym_\eta \left( 
    \frac{\partial_{\sigma^{-1}} \Delta_\eta}{\Delta_\eta} T(z)
  \right)
  = \sum_{\sigma \in S_\eta} f_\sigma D_\sigma^\eta T(z) \in L_\eta.  
\end{align*}
for each $z \in \ZZ^\mu_0$. Now let $R \subset S_\theta$ be a set of 
representatives of left $S_\eta$-coclasses of $S_\theta$, so $S_\theta
= \bigsqcup_{\sigma \in R} S_\eta \sigma$. Then
\begin{align*}
D_\tau^\theta T(z)
  &= \sum_{\sigma \in R} \sym_\eta \left(
    \frac{\sg(\sigma) \sigma(\partial_{\tau^{-1}} \Delta_\theta)
      (\Delta_\eta/\Delta_\theta)}
    {\Delta_\eta} T(\sigma(z))
  \right).
\end{align*}
Since $\sg(\sigma) \sigma(\partial_{\tau^{-1}} \Delta_\theta) \in \CC[X_\eta]$ 
and $\Delta_\eta/\Delta_\theta \in B_\eta$ we get that $D_\tau^\theta T(z) \in 
L_\eta$ and hence $L_\eta = B_\eta \ot_{B_\theta} L_\theta$, so
\begin{align*}
V(T(v)) 
  = \CC_v \ot_{B_\eta} L_\eta 
  = \CC_v \ot_{B_\eta} (B_\eta \ot_{B_\theta} L_\theta)
  = \CC_v \ot_{B_\theta} L_\theta
  = W(T(v)).
\end{align*}
\end{Remark}

\end{document}